\documentclass[12pt]{article}
\usepackage{amsthm,a4wide,amssymb,amsmath,enumerate,float}
\newtheorem{theorem}{Theorem}
\newtheorem{lemma}[theorem]{Lemma}
\newtheorem{conjecture}[theorem]{Conjecture}
\newtheorem{claim}{Claim}
\allowdisplaybreaks

\usepackage{lineno}

\title{\Large On the maximum number of minimum dominating sets in forests}
\author{\large
J.D. Alvarado$^1$\and  
S. Dantas$^1$\and 
E. Mohr$^2$\and 
D. Rautenbach$^2$}
\date{}

\begin{document}


\maketitle

\begin{center}
$^1$ Instituto de Matem\'{a}tica e Estat\'{i}stica, Universidade Federal Fluminense, Niter\'{o}i, Brazil, \texttt{josealvarado.mat17@gmail.com, sdantas@im.uff.br}\\[3mm]
$^2$Institut f\"{u}r Optimierung und Operations Research,
Universit\"{a}t Ulm, Ulm, Germany,
\{\texttt{elena.mohr, dieter.rautenbach}\}\texttt{@uni-ulm.de}\\[3mm]
\end{center}

\newcommand{\mds}{\sharp\gamma}

\begin{abstract}
Fricke, Hedetniemi, Hedetniemi, and Hutson asked whether every tree 
with domination number $\gamma$ has at most $2^\gamma$ 
minimum dominating sets.
Bie\'{n} gave a counterexample, 
which allows to construct forests with domination number $\gamma$
and $2.0598^\gamma$ minimum dominating sets.
We show that every forest with domination number $\gamma$ 
has at most $2.4606^\gamma$ minimum dominating sets,
and that every tree with independence number $\alpha$ 
has at most $2^{\alpha-1}+1$ maximum independent sets.
\end{abstract}
{\small
\begin{tabular}{lp{12.5cm}}
\textbf{Keywords:} & 
tree; domination number; minimum dominating set; 
independence number; maximum independent set
\end{tabular}
}

\pagebreak

\section{Introduction}

We consider only finite, simple, and undirected graphs.
A {\it dominating set} \cite{hahesl} of a graph $G$
is a subset $D$ of the vertex set $V(G)$ of $G$ 
such that every vertex in $V(G)\setminus D$ has a neighbor in $D$.
The {\it domination number} $\gamma(G)$ of $G$ 
is the minimum cardinality of a dominating set of $G$.
A dominating set of $G$ is 
{\it minimal} if no proper subset is dominating,
and 
{\it minimum} if it has cardinality $\gamma(G)$.
An {\it independent set} of a graph $G$
is a set of pairwise non-adjacent vertices of $G$.
The {\it independence number} $\alpha(G)$ of $G$ 
is the maximum cardinality of an independent set of $G$.
An independent set of $G$ is 
{\it maximal} if no proper superset is independent,
and 
{\it maximum} if it has cardinality $\alpha(G)$.

The motivation for the research reported in the present paper
was the question of Fricke et al.~\cite{frhehehu}
whether every tree with domination number $\gamma$
has at most $2^\gamma$ minimum dominating sets.
Bie\'{n} \cite{bi} pointed out 
that the tree $T^*$ in Figure \ref{fig1} is a counterexample.

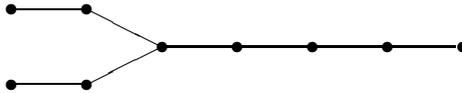
\begin{figure}[H]
\begin{center}
\unitlength 1mm 
\linethickness{0.4pt}
\ifx\plotpoint\undefined\newsavebox{\plotpoint}\fi 
\begin{picture}(66,16)(0,0)
\put(5,5){\circle*{1.5}}
\put(15,5){\circle*{1.5}}
\put(25,10){\circle*{1.5}}
\put(35,10){\circle*{1.5}}
\put(45,10){\circle*{1.5}}
\put(55,10){\circle*{1.5}}
\put(65,10){\circle*{1.5}}
\put(15,15){\circle*{1.5}}
\put(5,15){\circle*{1.5}}
\put(5,15){\line(1,0){10}}
\put(15,15){\line(2,-1){10}}
\put(25,10){\line(-2,-1){10}}
\put(15,5){\line(-1,0){10}}
\put(25,10){\line(1,0){10}}
\put(35,10){\line(1,0){29}}
\end{picture}
\end{center}
\vspace{-5mm}
\caption{A tree $T^*$ with domination number $4$ and $18$ minimum dominating sets.}\label{fig1}
\end{figure}
While the maximum number of minimum dominating sets 
of a tree with domination number $\gamma$ 
is more than $2^\gamma$,
we believe that $2^\gamma$ 
is still the correct exponential part;
more precisely, we pose the following conjecture.

\begin{conjecture}\label{conjecture1}
A tree with domination number $\gamma$ has at most 
$O\left(\frac{\gamma 2^{\gamma}}{\ln \gamma}\right)$
minimum dominating sets.
\end{conjecture}
A forest $F$ with domination number $\gamma$ 
whose components are copies of $T^*$ 
has $18^{\frac{\gamma}{4}}\approx 2.0598^\gamma$ minimum dominating sets,
that is, connectivity is essential for Conjecture \ref{conjecture1}.
Connolly et al.~\cite{cogagokake} showed that 
the maximum number of minimum dominating set 
of a graph of order $n$ and domination number $\gamma$ at least $3$ 
is between 
${n\choose \gamma}-O\left(n^{\gamma-1}\right)$
and
${n\choose \gamma}-\Omega\left(n^{\gamma-1-\frac{1}{\gamma-1}}\right)$,
that is, 
for general graphs, 
there is no upper bound on the number of minimum dominating sets
that only depends on the domination number.

The following is our main contribution.

\begin{theorem}\label{theorem1}
A forest with domination number $\gamma$ has at most 
$2.4606^\gamma$
minimum dominating sets.
\end{theorem}
There are some related results
concerning the maximum number of dominating sets 
that are not necessarily minimum.
Fomin et al.~\cite{fogrpyst} showed that 
a graph of order $n$ has at most $1.7696^n$ minimal dominating sets.
Br\'{o}d and Skupie\'{n} \cite{brsk} showed that trees of order $n$ 
can have up to $\Theta\left(5^{\frac{n}{3}}\right)$ dominating sets.
They actually determined the precise maximum number of dominating sets 
for every value of $n$ and characterized the extremal trees,
each of which has in fact at most four distinct minimum dominating sets.

Inspired by Conjecture \ref{conjecture1} and Theorem \ref{theorem1},
we consider the analogous question of how many 
maximum independent sets 
a tree with independence number $\alpha$ can have.
Since the independence number allows simpler reductions
than the domination number,
it is not surprising that 
the maximum number of maximum independent sets 
is better understood.
Zito \cite{zi} showed that the number of maximum independent sets
of a tree of order $n$ is at most
$2^{\frac{n-2}{2}}+1$ if $n$ is even,
and at most
$2^{\frac{n-3}{2}}$ if $n$ is odd,
and she also characterized all extremal trees.
For similar results concerning the maximum number of maximal independent sets see \cite{kogodo,wl}.

As our second contribution 
we give a very simple independent proof 
of the following variation of Theorem \ref{theorem1},
which can also be derived from Zito's results
using the fact that the independence number 
of every tree is at least half its order.

\begin{theorem}\label{theorem2}
A tree $T$ with independence number $\alpha$ has at most 
$2^{\alpha-1}+1$
maximum independent sets
with equality if and only if 
$T$ arises by subdividing $\alpha-1$ edges  
of a star of order $\alpha+1$
once.
\end{theorem}
The following section contains the proofs of our results.
Furthermore, we collect some properties of trees
with a fixed domination number having 
the maximum possible number of minimum dominating sets, and
present a candidate tree
motivating Conjecture \ref{conjecture1}.

\section{Proofs and comments}

We need some more terminology.

Let $F$ be a forest.
An {\it endvertex} of $F$ is a vertex of degree at most $1$.
A neighbor of an endvertex of $F$ is a {\it support vertex} of $F$.
If a support vertex of $F$ is adjacent to at least two endvertices,
then it is a {\it strong support vertex} of $F$.
Clearly, every minimum dominating set of $F$ 
contains every strong support vertex of $F$.

\begin{proof}[Proof of Theorem \ref{theorem1}]
For two positive integers $\gamma$ and $s$ with $s\leq \gamma$,
let $f(\gamma,s)$ be the maximum number of minimum dominating sets
among all forests 
with domination number $\gamma$ and $s$ strong support vertices.
Let $\beta$ be the largest solution of the equation $\beta^3-\beta^2-4\beta+1 = 0$,
that is, $\beta\approx 2.4606$.
Let $\alpha=\frac{\beta}{\beta-1}\approx 1.6847$.
By an inductive argument, 
we show $f(\gamma,s)\leq \alpha^s\beta^{\gamma-s}$.
Since $\alpha<\beta$, this implies the desired statement.
Clearly, $f(1,0)=2<\beta$ and $f(\gamma,\gamma)=1<\alpha^\gamma$.
Hence, we may assume that $\gamma\geq 2$ and $s<\gamma$.

For a contradiction, 
suppose that $\gamma$ and $s$ are chosen 
such that $f(\gamma,s)>\alpha^s\beta^{\gamma-s}$
and, subject to this condition, 
the value of $\gamma$ is as small as possible.
Let the forest $F$ 
with domination number $\gamma$ and $s$ strong support vertices
have $f(\gamma,s)$ minimum dominating sets.
If every component of $F$ has domination number $1$,
then $f(\gamma,s)\leq 2^{\gamma-s}<\alpha^s\beta^{\gamma-s}$.
Hence, we may assume that some component $T$ of $F$ 
has domination number at least $2$,
that is, the tree $T$ is not a star.
Let $v$ be a vertex of $T$ 
that is not an endvertex of $T$ 
and has exactly one neighbor $w$ that is not an endvertex of $T$.
Let $u$ be an endvertex of $T$ that is a neighbor of $v$.
In what follows, we construct several forests $F^{(1)},F^{(2)},\ldots$ derived from $F$,
and denote by $s^{(i)}$ the number of strong support vertices of $F^{(i)}$.

\begin{claim}\label{claim1}
$v$ is not a strong support vertex of $T$.
\end{claim}
\begin{proof}[Proof of Claim \ref{claim1}]
Suppose, for a contradiction, that $v$ is a strong support vertex.

First, we assume that $w$ is also a strong support vertex.
Let $F^{(1)}=F-(N_F[v]\setminus \{ w\})$.
$F^{(1)}$ has domination number $\gamma-1$ and $s-1$ strong support vertices.
A set $D$ is a minimum dominating set of $F$
if and only if $D=D'\cup \{ v\}$,
where $D'$ is a minimum dominating set of $F^{(1)}$.
By the choice of $\gamma$ and $s$, this implies 
$$f(\gamma,s)
\leq f(\gamma-1,s-1)
\leq \alpha^{s-1}\beta^{\gamma-s}
<\alpha^s\beta^{\gamma-s}.$$
Hence, we may assume that $w$ is no strong support vertex.

We will now use an important argument for the first time,
which we apply frequently within the rest of the proof.
As observed above, 
every minimum dominating set of $F$ contains $v$.
Therefore, if $F'=F-(N_F[v]\setminus \{ w\})$,
then we have to estimate 
the number of (almost) dominating sets $D'$ of $F'$ of cardinality $\gamma-1$,
where the vertex $w$ is already dominated (from the `outside' by $v$),
that is, the vertex $w$ may be used within $D'$ to dominate neighbors in $F'$
but no vertex in $D'$ needs to dominate $w$.
Now, such sets $D'$ that do not contain $w$
are simply dominating sets of cardinality $\gamma-1$ of $F'-w$,
while such sets $D'$ that contain $w$
are in one-to-one correspondence
with minimum dominating sets of cardinality $\gamma-1$ of 
the forest that arises from $F'$ by attaching two new endvertices to $w$.
This latter operation forces the minimum dominating set to contain $w$,
in which case it is irrelevant that $w$ is already dominated from the outside. 
More preciely, 
let $F^{(1)}$ arise from $F-(N_F[v]\setminus \{ w\})$ 
by attaching two new endvertices to $w$,
and let $F^{(2)}=F-N_F[v]$.
Note that $s^{(1)}=s$ and $s^{(2)}\geq s-1$.
A set $D$ is a minimum dominating set of $F$
if and only if $D=D'\cup \{ v\}$,
where $D'$ is a minimum dominating set of cardinality $\gamma-1$ 
of either $F^{(1)}$ or $F^{(2)}$.
Note that the domination number of $F^{(1)}$ and $F^{(2)}$ 
may be different from $\gamma-1$,
in which case 
the corresponding forest just has 
$0$ minimum dominating sets of cardinality $\gamma-1$.
By the choice of $\gamma$ and $s$, this implies 
\begin{eqnarray*}
f(\gamma,s) & \leq & 
f\left(\gamma-1,s^{(1)}\right)
+f\left(\gamma-1,s^{(2)}\right)\\
& \leq & 
\alpha^{s^{(1)}}\beta^{\gamma-s^{(1)}-1}
+\alpha^{s^{(2)}}\beta^{\gamma-s^{(2)}-1}\\
& \leq & 
\alpha^s\beta^{\gamma-s-1}
+\alpha^{s-1}\beta^{\gamma-s}\\
& = & \left(\frac{1}{\beta}+\frac{1}{\alpha}\right)\alpha^s\beta^{\gamma-s}\\
& = & \alpha^s\beta^{\gamma-s},
\end{eqnarray*}
where the last equality follows from the choice of $\alpha$.
This contradiction completes the proof of the claim.
\end{proof}

\begin{claim}\label{claim2}
$w$ is not a support vertex of $T$.
\end{claim}
\begin{proof}[Proof of Claim \ref{claim2}]
Suppose, for a contradiction, that $w$ is a support vertex.
Let $F^{(1)}=F-\{ u,v\}$.
$F^{(1)}$ has domination number $\gamma-1$ and $s$ strong support vertices.
A set $D$ is a minimum dominating set of $F$
if and only if 
$D\in \Big\{ D'\cup \{ v\}, D'\cup \{ u\}\Big\}$,
where $D'$ is a minimum dominating set of $F^{(1)}$.
By the choice of $\gamma$ and $s$, this implies 
$$f(\gamma,s)
\leq 2f(\gamma-1,s)
\leq 2\alpha^{s}\beta^{\gamma-s-1}
< \alpha^s\beta^{\gamma-s}.$$
This contradiction completes the proof of the claim.
\end{proof}
For the remaining claims, 
we restrict the choice of $u$, $v$, and $w$ slightly further.
More precisely, we assume that 
$T$ is rooted at some endvertex, 
$u$ is an endvertex of maximum depth,
$v$ is the parent of $u$, 
and $w$ is the parent of $v$.

\begin{claim}\label{claim3}
$w$ has exactly one child in $T$.
\end{claim}
\begin{proof}[Proof of Claim \ref{claim3}]
Suppose, for a contradiction, that $w$ has a child $v'$ distinct from $v$.
By Claims \ref{claim1} and \ref{claim2}, 
the vertex $v'$ has a unique child $u'$.
By the choice of $u$, the vertex $u'$ is an endvertex.
Let $F^{(1)}=F-\{ u,v,u',v'\}$,
let $F^{(2)}$ arise from $F^{(1)}$ by attaching two new endvertices to $w$, and
let $F^{(3)}=F-\{ u,v,w,u',v'\}$.
Note that 
$s^{(1)}\geq s$,
$s^{(2)}=s+1$, and
$s^{(3)}\geq s$.
A set $D$ is a minimum dominating set of $F$
if and only if 
\begin{itemize}
\item either $D=D'\cup \{ u,u'\}$,
where $D'$ is a minimum dominating set of cardinality $\gamma-2$ 
of $F^{(1)}$,
\item or $D\in \Big\{ D'\cup \{ u,v'\},D'\cup \{ u',v\},D'\cup \{ v,v'\}\Big\}$,
where $D'$ is a minimum dominating set of cardinality $\gamma-2$ of 
either $F^{(2)}$ or $F^{(3)}$.
\end{itemize}
By the choice of $\gamma$ and $s$, this implies 
\begin{eqnarray*}
f(\gamma,s) & \leq & 
f\left(\gamma-1,s^{(1)}\right)
+3f\left(\gamma-1,s^{(2)}\right)
+3f\left(\gamma-1,s^{(3)}\right)\\
& \leq & 
\alpha^{s^{(1)}}\beta^{\gamma-s^{(1)}-2}
+3\alpha^{s^{(2)}}\beta^{\gamma-s^{(2)}-2}
+3\alpha^{s^{(3)}}\beta^{\gamma-s^{(3)}-2}\\
& \leq & 
4\alpha^s\beta^{\gamma-s-2}
+3\alpha^{s+1}\beta^{\gamma-s-3}\\
& = & \left(\frac{4}{\beta^2}+\frac{3\alpha}{\beta^3}\right)
\alpha^s\beta^{\gamma-s}\\
& = & \alpha^s\beta^{\gamma-s},
\end{eqnarray*}
where the last equality follows from the choice of $\alpha$ and $\beta$.
This contradiction completes the proof of the claim.
\end{proof}
Since $T$ has domination number at least $2$,
the vertex $w$ has a parent $x$.

\begin{claim}\label{claim4}
$x$ is not a support vertex of $T$.
\end{claim}
\begin{proof}[Proof of Claim \ref{claim4}]
Suppose, for a contradiction, that $x$ is a support vertex.
Let $F^{(1)}=F-\{ u,v\}$,
and let $F^{(2)}=F-\{ u,v,w\}$.
Note that 
$s^{(1)}\geq s$, and
$s^{(2)}=s$.
A set $D$ is a minimum dominating set of $F$
if and only if 
\begin{itemize}
\item either $D=D'\cup \{ u\}$,
where $D'$ is a minimum dominating set of cardinality $\gamma-1$ 
of $F^{(1)}$,
\item or $D=D'\cup \{ v\}$,
where $D'$ is a minimum dominating set of cardinality $\gamma-1$ 
of $F^{(2)}$.
\end{itemize}
By the choice of $\gamma$ and $s$, this implies 
\begin{eqnarray*}
f(\gamma,s) & \leq & 
f\left(\gamma-1,s^{(1)}\right)
+f\left(\gamma-1,s^{(2)}\right)\\
& \leq & 
\alpha^{s^{(1)}}\beta^{\gamma-s^{(1)}-1}
+\alpha^{s^{(2)}}\beta^{\gamma-s^{(2)}-1}\\
& \leq & 2\alpha^s\beta^{\gamma-s-1}\\
& < & \alpha^s\beta^{\gamma-s}.
\end{eqnarray*}
This contradiction completes the proof of the claim.
\end{proof}

\begin{claim}\label{claim5}
$x$ has exactly one child in $T$.
\end{claim}
\begin{proof}[Proof of Claim \ref{claim5}]
Suppose, for a contradiction, that $x$ has a child $w'$ distinct from $w$.
By Claim \ref{claim4}, the vertex $w'$ has a child $v'$.

First, we assume that the vertex $v'$ has a child $u'$.
By the choice of $u$,
the vertices $u$ and $u'$ are endvertices of the same depth.
Therefore, by symmetry between $(u,v,w)$ and $(u',v',w')$, we obtain
that $v'$ is the only child of $w'$,
and that $u'$ is the only child of $v'$.
Let $F^{(1)}=F-\{ u,v,u',v'\}$,
let $F^{(2)}=F-\{ u,v,w,u',v'\}$,
let $F^{(3)}=F-\{ u,v,u',v',w'\}$, and
let $F^{(4)}=F-\{ u,v,w,u',v',w'\}$.
Note that 
$s^{(1)}\geq s+1$,
$s^{(2)}=s$,
$s^{(3)}=s$, and
$s^{(4)}\geq s$.
A set $D$ is a minimum dominating set of $F$
if and only if 
\begin{itemize}
\item either $D=D'\cup \{ u,u'\}$,
where $D'$ is a minimum dominating set of cardinality $\gamma-2$ 
of $F^{(1)}$,
\item or $D=D'\cup \{ v,u'\}$,
where $D'$ is a minimum dominating set of cardinality $\gamma-2$ 
of $F^{(2)}$,
\item or $D=D'\cup \{ u,v'\}$,
where $D'$ is a minimum dominating set of cardinality $\gamma-2$ 
of $F^{(3)}$,
\item or $D=D'\cup \{ v,v'\}$,
where $D'$ is a minimum dominating set of cardinality $\gamma-2$ 
of $F^{(4)}$.
\end{itemize}
By the choice of $\gamma$ and $s$, this implies 
\begin{eqnarray*}
f(\gamma,s) & \leq & 
f\left(\gamma-2,s^{(1)}\right)
+f\left(\gamma-2,s^{(2)}\right)
+f\left(\gamma-2,s^{(3)}\right)
+f\left(\gamma-2,s^{(4)}\right)\\
& \leq & 
\alpha^{s^{(1)}}\beta^{\gamma-s^{(1)}-2}
+\alpha^{s^{(2)}}\beta^{\gamma-s^{(2)}-2}
+\alpha^{s^{(3)}}\beta^{\gamma-s^{(3)}-2}
+\alpha^{s^{(4)}}\beta^{\gamma-s^{(4)}-2}\\
& \leq & 
3\alpha^{s}\beta^{\gamma-s-2}
+\alpha^{s+1}\beta^{\gamma-s-3}\\
& = & \left(\frac{3}{\beta^2}+\frac{\alpha}{\beta^3}\right)
\alpha^s\beta^{\gamma-s}\\
& < & \left(\frac{4}{\beta^2}+\frac{3\alpha}{\beta^3}\right)
\alpha^s\beta^{\gamma-s}\\
& = & \alpha^s\beta^{\gamma-s}.
\end{eqnarray*}
Hence, we may assume that every child of $w'$ is an endvertex.
By Claim \ref{claim1}, the vertex $v'$ is the only child of $w'$.
Let $F^{(1)}=F-\{ u,v,v',w'\}$,
let $F^{(2)}=F-\{ u,v,w,v',w'\}$,
let $F^{(3)}$ arise from $F^{(2)}$ by attaching two new endvertices to $x$, and
let $F^{(4)}=F-\{ u,v,w,x,v',w'\}$.
Note that 
$s^{(1)}=s$,
$s^{(2)}\geq s$,
$s^{(3)}=s+1$, and
$s^{(4)}\geq s$.
A set $D$ is a minimum dominating set of $F$
if and only if 
\begin{itemize}
\item either $D\in \Big\{ D'\cup \{ u,v'\},D'\cup \{ u,w'\}\Big\}$,
where $D'$ is a minimum dominating set of cardinality $\gamma-2$ 
of $F^{(1)}$,
\item or $D=D'\cup \{ v,v'\}$,
where $D'$ is a minimum dominating set of cardinality $\gamma-2$ 
of $F^{(2)}$,
\item or $D=D'\cup \{ v,w'\}$,
where $D'$ is a minimum dominating set of cardinality $\gamma-2$ 
of either $F^{(3)}$ or $F^{(4)}$.
\end{itemize}
By the choice of $\gamma$ and $s$, this implies 
\begin{eqnarray*}
f(\gamma,s) & \leq & 
2f\left(\gamma-2,s^{(1)}\right)
+f\left(\gamma-2,s^{(2)}\right)
+f\left(\gamma-2,s^{(3)}\right)
+f\left(\gamma-2,s^{(4)}\right)\\
& \leq & 
2\alpha^{s^{(1)}}\beta^{\gamma-s^{(1)}-2}
+\alpha^{s^{(2)}}\beta^{\gamma-s^{(2)}-2}
+\alpha^{s^{(3)}}\beta^{\gamma-s^{(3)}-2}
+\alpha^{s^{(4)}}\beta^{\gamma-s^{(4)}-2}\\
& \leq & 
4\alpha^{s}\beta^{\gamma-s-2}
+\alpha^{s+1}\beta^{\gamma-s-3}\\
& = & \left(\frac{4}{\beta^2}+\frac{\alpha}{\beta^3}\right)
\alpha^s\beta^{\gamma-s}\\
& < & \left(\frac{4}{\beta^2}+\frac{3\alpha}{\beta^3}\right)
\alpha^s\beta^{\gamma-s}\\
& = & \alpha^s\beta^{\gamma-s}.
\end{eqnarray*}
This contradiction completes the proof of the claim.
\end{proof}
By Claim \ref{claim2}, the vertex $x$ is not an endvertex.
Therefore, by Claims \ref{claim4} and \ref{claim5},
the vertex $x$ has a parent $y$.

\begin{claim}\label{claim6}
$y$ is not a support vertex of $T$.
\end{claim}
\begin{proof}[Proof of Claim \ref{claim6}]
Suppose, for a contradiction, that $y$ is a support vertex.
Let $x'$ be a neighbor of $y$ that is an endvertex.
For every minimum dominating set $D$ of $F$,
we obtain $D\cap \{ u,v,w,x,y,x'\}=\{ v,y\}$.
Let $F^{(1)}=F-\{ u,v,w\}$.
Note that $s^{(1)}\geq s$.
A set $D$ is a minimum dominating set of $F$
if and only if 
$D=D'\cup \{ v\}$,
where $D'$ is a minimum dominating set of cardinality $\gamma-1$ 
of $F^{(1)}$.
By the choice of $\gamma$ and $s$, this implies 
\begin{eqnarray*}
f(\gamma,s) & \leq & 
f\left(\gamma-1,s^{(1)}\right)
\leq 
\alpha^{s^{(1)}}\beta^{\gamma-s^{(1)}-1}
\leq
\alpha^{s}\beta^{\gamma-s-1}
<
\alpha^{s}\beta^{\gamma-s}.
\end{eqnarray*}
This contradiction completes the proof of the claim.
\end{proof}
We are now in a position to derive a final contradiction.
Let $F^{(1)}=F-\{ u,v,w,x\}$,
let $F^{(2)}=F-\{ u,v,w,x,y\}$, and
let $F^{(3)}=F-\{ u,v,w\}$.
Note that 
$s^{(1)}\geq s$,
$s^{(2)}\geq s$, and
$s^{(3)}=s$.
A set $D$ is a minimum dominating set of $F$
if and only if 
\begin{itemize}
\item either $D\in \Big\{ D'\cup \{ u,w\},D'\cup \{ v,w\}\Big\}$,
where $D'$ is a minimum dominating set of cardinality $\gamma-2$ 
of $F^{(1)}$,
\item or $D=D'\cup \{ u,x\}$,
where $D'$ is a minimum dominating set of cardinality $\gamma-2$ 
of $F^{(2)}$,
\item or $D=D'\cup \{ v\}$,
where $D'$ is a minimum dominating set of cardinality $\gamma-1$ 
of $F^{(3)}$.
\end{itemize}
By the choice of $\gamma$ and $s$, this implies 
\begin{eqnarray*}
f(\gamma,s) & \leq & 
2f\left(\gamma-2,s^{(1)}\right)
+f\left(\gamma-2,s^{(2)}\right)
+f\left(\gamma-1,s^{(3)}\right)\\
& \leq & 
2\alpha^{s^{(1)}}\beta^{\gamma-s^{(1)}-2}
+\alpha^{s^{(2)}}\beta^{\gamma-s^{(2)}-2}
+\alpha^{s^{(3)}}\beta^{\gamma-s^{(3)}-1}\\
& \leq & 
3\alpha^{s}\beta^{\gamma-s-2}
+\alpha^{s}\beta^{\gamma-s-1}\\
& = & \left(\frac{3}{\beta^2}+\frac{1}{\beta}\right)
\alpha^s\beta^{\gamma-s}\\
& < & \left(\frac{4}{\beta^2}+\frac{3\alpha}{\beta^3}\right)
\alpha^s\beta^{\gamma-s}\\
& = & \alpha^s\beta^{\gamma-s}.
\end{eqnarray*}
This completes the proof.
\end{proof}
Inspection of the proof of Theorem \ref{theorem1}
shows that Claim \ref{claim3} yields the largest lower bound on $\beta$.
Further minor improvements of the bound seem possible 
by a finer case analysis for the situation considered in this claim.

We proceed to the proof of our second result.

For a non-negative integer $k$,
let $T(k)$ arise by subdividing $k$ edges  
of a star of order $k+2$ once.
Note that $T(k)$ has independence number $k+1$.

\begin{proof}[Proof of Theorem \ref{theorem2}]
Within this proof, 
we call a graph {\it special}
if it is isomorphic to $T(k)$ for some non-negative integer $k$.
Suppose, for a contradiction, that the theorem is false,
and let $\alpha$ be the smallest value of the independence number 
for which it fails.
Let $T$ be a tree with independence number $\alpha$ 
such that 
\begin{itemize}
\item either $\sharp\alpha(T)>2^{\alpha-1}+1$
\item or $\sharp\alpha(T)=2^{\alpha-1}+1$
but $T$ is not special,
\end{itemize}
where $\sharp\alpha(T)$ denotes 
the number of maximum independent sets of $T$.

It is easy to see that $T$ 
is not special and has diameter at least $4$,
which implies $\alpha\geq 3$.
We root $T$ at an endvertex of a longest path in $T$.
Let $u$ be an endvertex of maximum depth in $T$,
let $v$ be the parent of $u$,
let $w$ be the parent of $v$, and
let $x$ be the parent of $w$.

\setcounter{claim}{0}

\begin{claim}\label{claim1b}
$v$ is not a strong support vertex of $T$.
\end{claim}
\begin{proof}[Proof of Claim \ref{claim1b}]
Suppose, for a contradiction,
that $v$ is a strong support vertex of $T$.
By the choice of $\alpha$, we obtain
\begin{eqnarray*}
\sharp\alpha(T)
&\leq &\sharp\alpha\big(T-(N_T[v]\setminus \{ w\})\big)\\
&\leq &2^{\alpha-(d_T(v)-1)-1}+1\\
&\leq &2^{\alpha-3}+1\\
&\leq &2^{\alpha-1}.
\end{eqnarray*}
This contradiction completes the proof of the claim.
\end{proof}
Let $w$ have $p$ children that are no endvertices
as well as $q$ children that are endvertices.
Let $Q$ be the set of children of $w$ that are endvertices,
and let $P$ be the set of all descendants of $w$ 
that are not in $Q$.
By Claim \ref{claim1b},
each child of $w$ in $P$ has exactly one child,
in particular, $|P|=2p$.
See Figure \ref{fig3}.

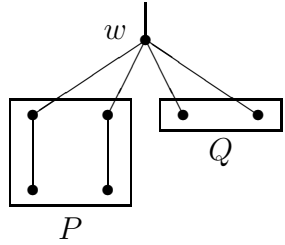
\begin{figure}[H]
\begin{center}
\unitlength 1mm 
\linethickness{0.4pt}
\ifx\plotpoint\undefined\newsavebox{\plotpoint}\fi 
\begin{picture}(38,30)(0,0)
\put(5,5){\circle*{1.5}}
\put(15,5){\circle*{1.5}}
\put(25,15){\circle*{1.5}}
\put(35,15){\circle*{1.5}}
\put(20,25){\circle*{1.5}}
\put(15,15){\circle*{1.5}}
\put(5,15){\circle*{1.5}}
\put(20,25){\line(0,1){5}}
\put(20,25){\line(1,-2){5}}
\put(35,15){\line(-3,2){15}}
\put(20,25){\line(-1,-2){5}}
\put(15,15){\line(0,-1){10}}
\put(20,25){\line(-3,-2){15}}
\put(5,15){\line(0,-1){10}}
\put(22,13){\framebox(16,4)[cc]{}}
\put(2,3){\framebox(16,14)[cc]{}}
\put(30,10){\makebox(0,0)[cc]{$Q$}}
\put(10,0){\makebox(0,0)[cc]{$P$}}
\put(16,26){\makebox(0,0)[cc]{$w$}}
\end{picture}
\end{center}
\caption{The vertex $w$ and its descendants.}\label{fig3}
\end{figure}
Let $T'=T-(\{ w\}\cup P\cup Q)$.

\begin{claim}\label{claim2b}
$q\in \{ 0,1\}$.
\end{claim}
\begin{proof}[Proof of Claim \ref{claim2b}]
Suppose, for a contradiction, that $q\geq 2$.
Since $Q$ is a subset of every maximum independent set of $T$,
and there are $2^p$ different ways 
in which a maximum independent set of $T$ 
can intersect $P$, we obtain, by the choice of $\alpha$,
\begin{eqnarray*}
\sharp\alpha(T)
&\leq &2^p\cdot \sharp\alpha(T')\\
&\leq &2^p\left(2^{\alpha-p-q-1}+1\right)\\
&\leq &2^{\alpha-q-1}+2^p\\
&\leq &2^{\alpha-3}+2^{\alpha-3}\\
&< &2^{\alpha-1}.
\end{eqnarray*}
This contradiction completes the proof of the claim.
\end{proof}

\begin{claim}\label{claim3b}
$q=0$.
\end{claim}
\begin{proof}[Proof of Claim \ref{claim3b}]
Suppose, for a contradiction, that $q=1$.
Note that $\alpha\geq p+2$.
Furthermore, 
if $T'$ is special,
then, since $T$ is not special,
the tree $T'$ is isomorphic to $T(k)$ for some $k\geq 1$,
which implies $\alpha\geq p+3$.
Note that the unique vertex in $Q$ is isolated in $T-(\{ w\}\cup P)$,
which implies $\sharp\alpha\big(T-(\{ w\}\cup P)\big)=\sharp\alpha(T')$.

If $T'$ is special,
then, by the choice of $\alpha$, we obtain
\begin{eqnarray*}
\sharp\alpha(T)
&\leq &
\sharp\alpha\big(T-P\big)
+\left(2^p-1\right)\cdot \sharp\alpha\big(T-(\{ w\}\cup P)\big)\\
&=&
\sharp\alpha\big(T-P\big)
+\left(2^p-1\right)\cdot \sharp\alpha(T')\\
&\leq &
\left(2^{\alpha-p-1}+1\right)
+\left(2^p-1\right)\left(2^{\alpha-p-2}+1\right)\\
&=& 2^{\alpha-2}+2^{\alpha-p-2}+2^p\\
&\leq & 2^{\alpha-2}+2^{\alpha-3}+2\\
&\leq & 2^{\alpha-1},
\end{eqnarray*}
where we use 
the strict convexity of the function $x\mapsto 2^x$,
and $\alpha\geq p+3\geq 4$.

Similarly, 
if $T'$ is not special,
then, by the choice of $\alpha$, we obtain
\begin{eqnarray*}
\sharp\alpha(T)
&\leq &
\sharp\alpha\big(T-P\big)
+\left(2^p-1\right)\cdot \sharp\alpha(T')\\
&\leq &
\left(2^{\alpha-p-1}+1\right)
+\left(2^p-1\right)2^{\alpha-p-2}\\
&=& 2^{\alpha-2}+2^{\alpha-p-2}+1\\
&\leq & 2^{\alpha-2}+2^{\alpha-3}+1\\
&\leq & 2^{\alpha-1},
\end{eqnarray*}
where we use $\alpha\geq 3$.
This contradiction completes the proof of the claim.
\end{proof}

\begin{claim}\label{claim4b}
$T'$ is not special.
\end{claim}
\begin{proof}[Proof of Claim \ref{claim4b}]
Suppose, for a contradiction, 
that $T'$ is isomorphic to $T(k)$ for some $k\geq 0$.
Figure \ref{fig3b} illustrates the four possibilities 
for the structure of $T$
together with the resulting value of $\sharp\alpha(T)$.
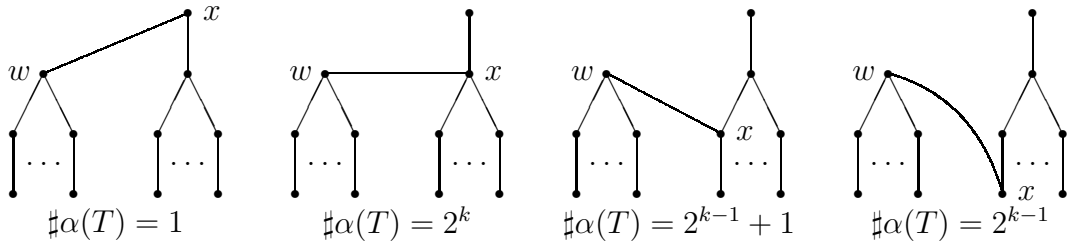
\begin{figure}[H]
\begin{center}
$\mbox{}$\hfill
\unitlength 0.8mm 
\linethickness{0.4pt}
\ifx\plotpoint\undefined\newsavebox{\plotpoint}\fi 
\begin{picture}(35,34)(0,0)
\put(0,3){\circle*{1.5}}
\put(24,3){\circle*{1.5}}
\put(10,3){\circle*{1.5}}
\put(34,3){\circle*{1.5}}
\put(10,13){\circle*{1.5}}
\put(34,13){\circle*{1.5}}
\put(0,13){\circle*{1.5}}
\put(24,13){\circle*{1.5}}
\put(5,23){\circle*{1.5}}
\put(29,23){\circle*{1.5}}
\put(29,33){\circle*{1.5}}
\put(10,13){\line(0,-1){10}}
\put(34,13){\line(0,-1){10}}
\put(0,13){\line(0,-1){10}}
\put(24,13){\line(0,-1){10}}
\put(0,13){\line(1,2){5}}
\put(24,13){\line(1,2){5}}
\put(5,23){\line(1,-2){5}}
\put(29,23){\line(1,-2){5}}
\put(5,8){\makebox(0,0)[cc]{$\ldots$}}
\put(29,8){\makebox(0,0)[cc]{$\ldots$}}
\put(29,33){\line(0,-1){10}}
\put(1,23){\makebox(0,0)[cc]{$w$}}
\put(17,-2){\makebox(0,0)[cc]{$\sharp\alpha(T)=1$}}
\put(33,33){\makebox(0,0)[cc]{$x$}}
\multiput(5,23)(.0808080808,.0336700337){297}{\line(1,0){.0808080808}}
\end{picture}\hfill
\begin{picture}(35,34)(0,0)
\put(0,3){\circle*{1.5}}
\put(24,3){\circle*{1.5}}
\put(10,3){\circle*{1.5}}
\put(34,3){\circle*{1.5}}
\put(10,13){\circle*{1.5}}
\put(34,13){\circle*{1.5}}
\put(0,13){\circle*{1.5}}
\put(24,13){\circle*{1.5}}
\put(5,23){\circle*{1.5}}
\put(29,23){\circle*{1.5}}
\put(29,33){\circle*{1.5}}
\put(10,13){\line(0,-1){10}}
\put(34,13){\line(0,-1){10}}
\put(0,13){\line(0,-1){10}}
\put(24,13){\line(0,-1){10}}
\put(0,13){\line(1,2){5}}
\put(24,13){\line(1,2){5}}
\put(5,23){\line(1,-2){5}}
\put(29,23){\line(1,-2){5}}
\put(5,8){\makebox(0,0)[cc]{$\ldots$}}
\put(29,8){\makebox(0,0)[cc]{$\ldots$}}
\put(29,33){\line(0,-1){10}}
\put(1,23){\makebox(0,0)[cc]{$w$}}
\put(17,-2){\makebox(0,0)[cc]{$\sharp\alpha(T)=2^k$}}
\put(33,23){\makebox(0,0)[cc]{$x$}}
\put(5,23){\line(1,0){24}}
\end{picture}\hfill
\begin{picture}(35,34)(0,0)
\put(0,3){\circle*{1.5}}
\put(24,3){\circle*{1.5}}
\put(10,3){\circle*{1.5}}
\put(34,3){\circle*{1.5}}
\put(10,13){\circle*{1.5}}
\put(34,13){\circle*{1.5}}
\put(0,13){\circle*{1.5}}
\put(24,13){\circle*{1.5}}
\put(5,23){\circle*{1.5}}
\put(29,23){\circle*{1.5}}
\put(29,33){\circle*{1.5}}
\put(10,13){\line(0,-1){10}}
\put(34,13){\line(0,-1){10}}
\put(0,13){\line(0,-1){10}}
\put(24,13){\line(0,-1){10}}
\put(0,13){\line(1,2){5}}
\put(24,13){\line(1,2){5}}
\put(5,23){\line(1,-2){5}}
\put(29,23){\line(1,-2){5}}
\put(5,8){\makebox(0,0)[cc]{$\ldots$}}
\put(29,8){\makebox(0,0)[cc]{$\ldots$}}
\put(29,33){\line(0,-1){10}}
\put(1,23){\makebox(0,0)[cc]{$w$}}
\put(17,-2){\makebox(0,0)[cc]{$\sharp\alpha(T)=2^{k-1}+1$}}
\put(28,13){\makebox(0,0)[cc]{$x$}}
\multiput(5,23)(.063973064,-.0336700337){297}{\line(1,0){.063973064}}
\end{picture}\hfill
\begin{picture}(35,34)(0,0)
\put(0,3){\circle*{1.5}}
\put(24,3){\circle*{1.5}}
\put(10,3){\circle*{1.5}}
\put(34,3){\circle*{1.5}}
\put(10,13){\circle*{1.5}}
\put(34,13){\circle*{1.5}}
\put(0,13){\circle*{1.5}}
\put(24,13){\circle*{1.5}}
\put(5,23){\circle*{1.5}}
\put(29,23){\circle*{1.5}}
\put(29,33){\circle*{1.5}}
\put(10,13){\line(0,-1){10}}
\put(34,13){\line(0,-1){10}}
\put(0,13){\line(0,-1){10}}
\put(24,13){\line(0,-1){10}}
\put(0,13){\line(1,2){5}}
\put(24,13){\line(1,2){5}}
\put(5,23){\line(1,-2){5}}
\put(29,23){\line(1,-2){5}}
\put(5,8){\makebox(0,0)[cc]{$\ldots$}}
\put(29,8){\makebox(0,0)[cc]{$\ldots$}}
\put(29,33){\line(0,-1){10}}
\put(1,23){\makebox(0,0)[cc]{$w$}}
\put(17,-2){\makebox(0,0)[cc]{$\sharp\alpha(T)=2^{k-1}$}}
\put(28,3){\makebox(0,0)[cc]{$x$}}
\qbezier(5,23)(19.5,19)(24,3)
\end{picture}
\hfill$\mbox{}$
\end{center}
\caption{The four possible configurations in the proof of Claim \ref{claim4b} together with the resulting values of $\sharp\alpha(T)$.}\label{fig3b}
\end{figure}
Since $p\geq 1$ and $\alpha=p+k+2$,
we have 
$2^k\leq 2^{\alpha-3}$ 
and 
$2^{k-1}+1\leq 2^{\alpha-4}+1$, that is, 
in all cases $\sharp\alpha(T)$ is less than $2^{\alpha-1}$.
This contradiction completes the proof of the claim.
\end{proof}
Let $T''=T-P$.

\begin{claim}\label{claim5b}
$T''$ is not special.
\end{claim}
\begin{proof}[Proof of Claim \ref{claim5b}]
Suppose, for a contradiction, 
that $T''$ is isomorphic to $T(k)$ for some $k\geq 0$.
Figure \ref{fig4b} illustrates the two possibilities 
for the structure of $T$
together with the resulting value of $\sharp\alpha(T)$.
Since $T$ is not special, we obtain $k\geq 1$.

\begin{figure}[H]
\begin{center}
$\mbox{}$\hfill
\unitlength 1mm 
\linethickness{0.4pt}
\ifx\plotpoint\undefined\newsavebox{\plotpoint}\fi 
\begin{picture}(38,38)(0,0)
\put(3,7){\circle*{1}}
\put(27,7){\circle*{1}}
\put(13,7){\circle*{1}}
\put(37,7){\circle*{1}}
\put(13,17){\circle*{1}}
\put(37,17){\circle*{1}}
\put(3,17){\circle*{1}}
\put(27,17){\circle*{1}}
\put(13,17){\line(0,-1){10}}
\put(37,17){\line(0,-1){10}}
\put(3,17){\line(0,-1){10}}
\put(27,17){\line(0,-1){10}}
\put(32,27){\circle*{1}}
\put(27,17){\line(1,2){5}}
\put(32,27){\line(1,-2){5}}
\put(8,12){\makebox(0,0)[cc]{$\ldots$}}
\put(32,12){\makebox(0,0)[cc]{$\ldots$}}
\put(32,37){\circle*{1}}
\put(32,37){\line(0,-1){10}}
\put(36,37){\makebox(0,0)[cc]{$w$}}
\put(0,5){\framebox(16,14)[cc]{}}
\put(3,22){\makebox(0,0)[cc]{$P$}}

\multiput(13,17)(.0336879433,.0354609929){564}{\line(0,1){.0354609929}}
\multiput(32,37)(-.0489038786,-.0337268128){593}{\line(-1,0){.0489038786}}
\put(20,0){\makebox(0,0)[cc]{$\sharp\alpha(T)=2^p+2^k$}}
\end{picture}\hfill
\begin{picture}(38,38)(0,0)
\put(3,7){\circle*{1}}
\put(27,7){\circle*{1}}
\put(13,7){\circle*{1}}
\put(37,7){\circle*{1}}
\put(13,17){\circle*{1}}
\put(37,17){\circle*{1}}
\put(3,17){\circle*{1}}
\put(27,17){\circle*{1}}
\put(13,17){\line(0,-1){10}}
\put(37,17){\line(0,-1){10}}
\put(3,17){\line(0,-1){10}}
\put(27,17){\line(0,-1){10}}
\put(32,27){\circle*{1}}
\put(27,17){\line(1,2){5}}
\put(32,27){\line(1,-2){5}}
\put(8,12){\makebox(0,0)[cc]{$\ldots$}}
\put(32,12){\makebox(0,0)[cc]{$\ldots$}}
\put(32,37){\circle*{1}}
\put(32,37){\line(0,-1){10}}
\put(31,7){\makebox(0,0)[cc]{$w$}}
\put(0,5){\framebox(16,14)[cc]{}}
\put(3,22){\makebox(0,0)[cc]{$P$}}

\qbezier(13,7)(22,7)(27,7)
\qbezier(27,7)(15,-1)(3,7)
\put(20,0){\makebox(0,0)[cc]{$\sharp\alpha(T)=1+2^{k-1}+2^{p+k-1}$}}
\end{picture}
\hfill$\mbox{}$
\end{center}
\caption{The two possible configurations in the proof of Claim \ref{claim5b} together with the resulting values of $\sharp\alpha(T)$.}\label{fig4b}
\end{figure}
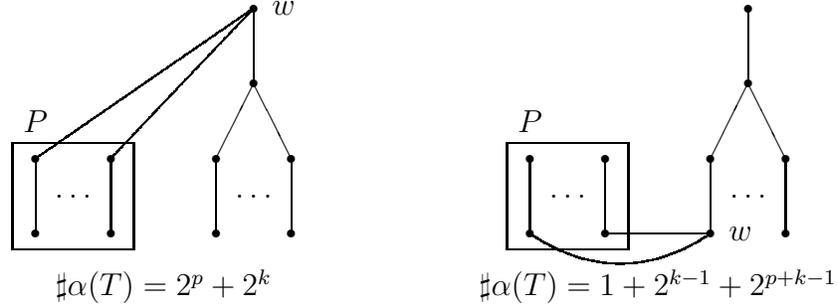
Since $p,k\geq 1$ and $\alpha=p+k+1\geq 3$,
we obtain
$2^p+2^k\leq 2^{p+k-1}+2\leq 2^{\alpha-2}+2\leq 2^{\alpha-1}$
and
$1+2^{k-1}+2^{p+k-1}
\leq 1+2^{\alpha-3}+2^{\alpha-2}
\leq 2^{\alpha-1}$,
that is, in both cases $\sharp\alpha(T)$ is at most $2^{\alpha-1}$.
This contradiction completes the proof of the claim.
\end{proof}
We are now in a position to derive a final contradiction.
By Claims \ref{claim4b} and \ref{claim5b},
and the choice of $\alpha$,
we obtain
\begin{eqnarray*}
\sharp\alpha(T)
&\leq &
\sharp\alpha\big(T''\big)
+\left(2^p-1\right)\cdot \sharp\alpha(T')\\
&\leq &
2^{\alpha-p-1}
+\left(2^p-1\right)2^{\alpha-p-1}\\
&=& 2^{\alpha-1}.
\end{eqnarray*}
This contradiction completes the proof.
\end{proof}
We proceed to our final results,
motivating Conjecture \ref{conjecture1}.

\begin{lemma}\label{lemma1}
Let $\gamma$ be a positive integer.
If $T$ is a tree such that 
\begin{enumerate}[(a)]
\item $T$ has domination number $\gamma$,
\item subject to condition (a), 
the tree $T$ has as many minimum dominating sets as possible,
\item subject to conditions (a) and (b), 
the tree $T$ has as few vertices as possible,
\end{enumerate}
then the following statements hold.
\begin{enumerate}[(i)]
\item Every endvertex of $T$ belongs to some minimum dominating set of $T$.
\item If $x,w(i),v(i,j),u(i,j)\in V(T)$ for $i\in [k]$ and $j\in [p_i]$
are such that (see Figure \ref{fig2})
\begin{itemize}
\item $u(i,j)$ is an endvertex for $i\in [k]$ and $j\in [p_i]$,
\item $N_T(v(i,j))=\{ u(i,j),w(i)\}$ for $i\in [k]$ and $j\in [p_i]$, and
\item $w(i)\in N_T(x)$ for $i\in [k]$, 
\end{itemize}
then $|p_r-p_s|\leq 1$ for $r,s\in [k]$.
\end{enumerate}
\end{lemma}
\begin{proof}
{\it (i)} Suppose, for a contradiction, that some endvertex $u$ of $T$
does not belong to any minimum dominating set of $T$.
Let $v$ be the neighbor of $u$ in $T$.
Let $T'=T-u$.
Since every minimum dominating set of $T$ is a dominating set of $T'$,
we have $\gamma(T')\leq \gamma(T)$.
If $\gamma(T')<\gamma(T)$, 
and $D'$ is a minimum dominating set of $T'$,
then $D'\cup\{ u\}$ is a minimum dominating set of $T$
that contains $u$, 
which is a contradiction.
Hence, we obtain that $\gamma(T')=\gamma(T)$,
and that a set $D$ is a minimum dominating set of $T$
if and only if it is is a minimum dominating set of $T'$,
which implies that $T$ and $T'$ have the same number of minimum dominating sets.
Since $T'$ has less vertices than $T$,
we obtain a contradiction to the choice of $T$.

\medskip

\noindent {\it (ii)} It suffices to consider the case $k=2$, and $p_1\leq p_2-2$.
Let 
$$U=\{ v(i,j):i\in [2]\mbox{ and }j\in [p_i]\}\cup \{ u(i,j):i\in [2]\mbox{ and }j\in [p_i]\}.$$
Let ${\cal D}$ be the set of minimum dominating sets of $T$.
For $S\subseteq \{ w(1),w(2),x\}$, let 
$$\sharp\gamma(S)=
\Big|
\Big\{
D\setminus U:
D\in {\cal D}\mbox{ with }D\cap \{ w(1),w(2),x\}=S
\Big\}
\Big|.$$
Note that $\sharp\gamma(S)=0$
unless $S\in \Big\{ \emptyset,\{ w(1)\},\{ w(2)\},\{ x\}\Big\}$.
Furthermore, 
it is easy to see that $\sharp\gamma(\{ w(1)\})=\sharp\gamma(\{ w(2)\})$.
Now, the number of minimum dominating sets of $T$ equals
\begin{eqnarray*}
&& \sharp\gamma(\{ x\})2^{p_1+p_2}\\
&& +\sharp\gamma(\{ w(1)\})2^{p_1}\left(2^{p_2}-1\right)
+\sharp\gamma(\{ w(2)\})\left(2^{p_1}-1\right)2^{p_2}\\
&&+\sharp\gamma(\emptyset)\left(2^{p_1}-1\right)\left(2^{p_2}-1\right)\\
&=& \Big(\sharp\gamma(\{ x\})+3\sharp\gamma(\{ w(1)\})+\sharp\gamma(\emptyset)\Big)2^{p_1+p_2}
-\Big(\sharp\gamma(\{ w(1)\})+\sharp\gamma(\emptyset)\Big)
\Big(2^{p_1}+2^{p_2}\Big)
+\sharp\gamma(\emptyset).
\end{eqnarray*}
Since, for a fixed value of $p_1+p_2$,
the expression $2^{p_1}+2^{p_2}$ 
is minimized choosing $p_1$ and $p_2$ as equal as possible,
the statement follows.
\end{proof}

\begin{figure}[H]
\begin{center}
{\footnotesize
\unitlength 1.5mm 
\linethickness{0.4pt}
\ifx\plotpoint\undefined\newsavebox{\plotpoint}\fi 
\begin{picture}(64,35)(0,0)
\put(5,5){\circle*{1}}
\put(45,5){\circle*{1}}
\put(5,15){\circle*{1}}
\put(45,15){\circle*{1}}
\put(5,5){\line(0,1){10}}
\put(45,5){\line(0,1){10}}
\put(5,15){\line(1,1){7}}
\put(45,15){\line(1,1){7}}
\put(12,22){\line(1,-1){7}}
\put(52,22){\line(1,-1){7}}
\put(19,15){\circle*{1}}
\put(59,15){\circle*{1}}
\put(19,5){\circle*{1}}
\put(59,5){\circle*{1}}
\put(0,5){\makebox(0,0)[cc]{$u(1,1)$}}
\put(40,5){\makebox(0,0)[cc]{$u(k,1)$}}
\put(24,5){\makebox(0,0)[cc]{$u(1,p_1)$}}
\put(64,5){\makebox(0,0)[cc]{$u(k,p_k)$}}
\put(12,22){\circle*{1}}
\put(52,22){\circle*{1}}
\put(32,32){\circle*{1}}
\put(9,22){\makebox(0,0)[cc]{$w(1)$}}
\put(55,22){\makebox(0,0)[cc]{$w(k)$}}
\put(0,15){\makebox(0,0)[cc]{$v(1,1)$}}
\put(40,15){\makebox(0,0)[cc]{$v(k,1)$}}
\put(24,15){\makebox(0,0)[cc]{$v(1,p_1)$}}
\put(64,15){\makebox(0,0)[cc]{$v(k,p_k)$}}
\put(32,15){\makebox(0,0)[cc]{$\ldots$}}
\put(32,5){\makebox(0,0)[cc]{$\ldots$}}
\put(52,15){\makebox(0,0)[cc]{$\ldots$}}
\put(52,5){\makebox(0,0)[cc]{$\ldots$}}
\put(12,15){\makebox(0,0)[cc]{$\ldots$}}
\put(12,5){\makebox(0,0)[cc]{$\ldots$}}
\put(19,15){\line(0,-1){10}}
\put(59,15){\line(0,-1){10}}
\put(52,22){\line(-2,1){20}}
\put(32,32){\line(-2,-1){20}}
\put(32,35){\makebox(0,0)[cc]{$x$}}
\end{picture}
}
\end{center}
\caption{The tree $T(p_1,\ldots,p_k)$.}\label{fig2}
\end{figure}
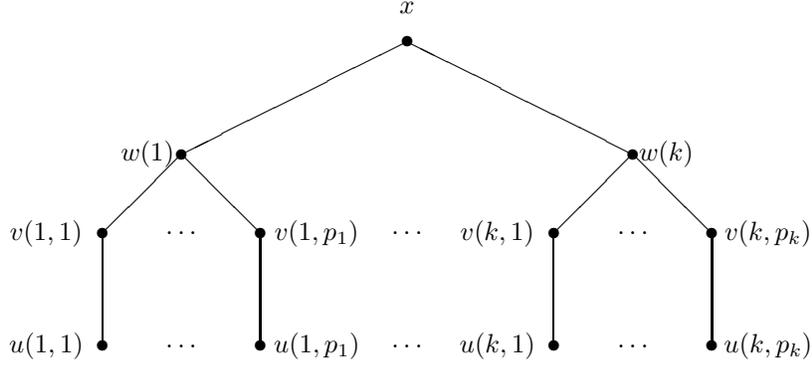

Let $\gamma$ and $k$ be positive integers with $k<\gamma$.
If $p_1,\ldots,p_k$ are such that $p_1+\cdots+p_k=\gamma-1$,
then the tree $T(p_1,\ldots,p_k)$ in Figure \ref{fig2}
has domination number $\gamma$.
Arguing as in the proof of \ref{lemma1}(ii), 
the number of minimum dominating sets is maximized by choosing 
$$p_i=
\begin{cases}
\left\lceil\frac{\gamma-1}{k}\right\rceil & \mbox{, if $i\in [(\gamma-1)\,\,{\rm mod}\,\,k]$, and}\\[3mm]
\left\lfloor\frac{\gamma-1}{k}\right\rfloor & \mbox{, if $i\in [k]\setminus [(\gamma-1)\,\,{\rm mod}\,\,k]$,} 
\end{cases}$$
in which case $T(p_1,\ldots,p_k)$ has exactly
\begin{eqnarray*}
&& 2^{\gamma-1}\\
&+&\big((\gamma-1)\,\,{\rm mod}\,\,k\big)
2^{\left\lceil\frac{\gamma-1}{k}\right\rceil}
\left(2^{\left\lceil\frac{\gamma-1}{k}\right\rceil}-1\right)^{\big((\gamma-1)\,\,{\rm mod}\,\,k\big)-1}
\left(2^{\left\lfloor\frac{\gamma-1}{k}\right\rfloor}-1\right)^{k-\big((\gamma-1)\,\,{\rm mod}\,\,k\big)}\\
&+&\big(k-\big((\gamma-1)\,\,{\rm mod}\,\,k\big)\big)
\left(2^{\left\lceil\frac{\gamma-1}{k}\right\rceil}-1\right)^{\big((\gamma-1)\,\,{\rm mod}\,\,k\big)}
2^{\left\lfloor\frac{\gamma-1}{k}\right\rfloor}
\left(2^{\left\lfloor\frac{\gamma-1}{k}\right\rfloor}-1\right)^{k-\big((\gamma-1)\,\,{\rm mod}\,\,k\big)-1}
\end{eqnarray*}
minimum dominating sets.
We believe that choosing $k\in \{ 1,2,\ldots,\gamma-1\}$ 
for a given positive integer $\gamma$ in such a way that 
this last expression is maximized, 
yields a tree with domination number $\gamma$
and the maximum number of minimum dominating sets.

The following table illustrates some explicit values
with the optimal choices of $k$.
\begin{figure}[H]
\begin{center}
\begin{tabular}{|c||c|c|c|c|}
\hline
$\gamma$ & 10 & 50 & 100 & 500\\
\hline
$\sharp\gamma$ & 1176 & $4160\cdot 10^{12}$ & $8187\cdot 10^{27}$ & $8152\cdot 10^{148}$\\
\hline
k& 3 & 9 & 16 & 56\\
\hline
\end{tabular}
\end{center}
\caption{The maximum number $\sharp\gamma$ of minimum dominating sets of a tree $T(p_1,\ldots,p_k)$ with domination number $\gamma$, and the corresponding value of $k$.}\label{fig3}
\end{figure}
Asymptotic considerations suggest that the optimum value of $k$ grows like
$\Theta\left(\frac{\gamma}{\log \gamma}\right)$, 
which motivates Conjecture \ref{conjecture1}.

\end{document}